\documentclass[a4paper,10pt]{article}

\usepackage[utf8]{inputenc}
\usepackage[T1]{fontenc} 
\usepackage[ngerman,english]{babel} 
\usepackage{mathtools}
\usepackage{amsmath}
\usepackage{lmodern} 
\usepackage{amssymb}
\usepackage{amsthm}
\usepackage{threeparttable}
\usepackage{caption}
\usepackage{float}
\usepackage{footnote}
\usepackage{url}
\usepackage{hyperref}
\usepackage{tabularx}
\usepackage{subcaption}
\usepackage{enumitem}
\usepackage{mwe}
\usepackage[toc,page]{appendix}
\usepackage{nicefrac}
\usepackage{wasysym}
\usepackage{a4wide}

\numberwithin{equation}{section}
\theoremstyle{definition}
\newtheorem{rem}{Remark}[section]
\theoremstyle{definition}
\newtheorem{defi}{Definition}[section]
\theoremstyle{definition}
\newtheorem{lem}{Lemma}[section]
\theoremstyle{definition}
\newtheorem{cor}{Corollary}[section]
\theoremstyle{definition}
\newtheorem{theo}[lem]{Theorem}

\theoremstyle{definition}

\usepackage{pdfcomment}
\usepackage{tikz}
\usepackage{pgf}
\usetikzlibrary{arrows}
\usetikzlibrary{arrows,automata}
\usetikzlibrary{positioning,arrows}
\tikzset{
	state/.style={circle,draw,minimum size=6ex},
	arrow/.style={-latex, shorten >=0.1ex, shorten <=0.1ex}}

\usetikzlibrary{arrows,decorations.markings}

\begin{document}

\title{Feedback Stabilization for a coupled PDE-ODE Production System}

\author{Vanessa Baumgärtner\footnotemark[1], \; Simone  Göttlich\footnotemark[1], \; Stephan Knapp\footnotemark[1]}
\renewcommand*{\thefootnote}{\fnsymbol{footnote}}

\footnotetext[1]{University of Mannheim, Department of Mathematics, 68131 Mannheim, Germany (goettlich@uni-mannheim.de).}

\date{March 26, 2019}

\maketitle

\begin{abstract}
We consider an interlinked production model consisting of conservation laws (PDE) coupled to ordinary differential equations (ODE).
Our focus is the analysis of control laws for the coupled system and corresponding stabilization questions of equilibrium dynamics in the 
presence of disturbances. These investigations are carried out using an appropriate Lyapunov function on the theoretical and numerical level. 
The discrete $L^2-$stabilization technique allows to derive a mixed feedback law that is able to ensure exponential stability 
also in bottleneck situations. All results are accompanied by computational examples.
\end{abstract}


\section{Introduction}
Mathematical models to describe production systems have gained a lot of attention
during the past few decades. In particular, fluid-like models based on scalar hyperbolic conservation laws (PDE)
have been developed and analyzed, see for example 
\cite{ArmbrusterDegondRinghofer2006aa,ArmbrusterRinghofer2005aa,Bretti2007,CoronKawskiWang2010aa,CoronWang2012aa,DApice2010,DApice2006}. 
The extension to networks has been originally introduced in \cite{Goettlich2006}, wherein, as coupling conditions, queues in terms of
ordinary differential equations (ODE) have been installed. This approach immediately leads to a coupled PDE-ODE dynamics, see \cite{DApice2010} 
for more details. 

In this article, we are concerned with the stabilization of such a coupled PDE-ODE system.
The dynamics within each processing unit is given by the linear advection equation with positive velocity while
queues in front of the unit are introduced to avoid congestion. Therefore, the queues measure the difference between the in-and outflow
into a processing unit and therefore obey a nonlinear ordinary differential equation.  
In recent years, the stability of linear coupled systems \cite{Barreau,Prieur,Krstic2008,Tang2011,Wu2014} has been investigated intensively.
However, there is only little literature available on the stability analysis for nonlinear coupled systems, see for example \cite{Wu2013a}.
Therefore, our purpose is to focus on the stability analysis of the coupled PDE-ODE production system from a theoretical as well as
discrete (or numerical) point of view. 
Due to the fact that the coupled PDE-ODE system can be interpreted as a coupled boundary value problem,
stabilization results for hyperbolic equations can be applied and extended, see \cite{Bastin2016,Coron2007aa,Li2010aa}.
Therein, analytical results for sufficiently smooth solutions in connection with boundary control have been obtained. 
The underlying tools for the study of those problems are Lyapunov functions stabilizing the deviation from steady states in suitable norms, e.g.~$L^2$ in space. Exponential decay of a continuous Lyapunov function under a so-called dissipative boundary condition has been proven in \cite{L4,CorondAndrea-NovelBastin2007aa}.
Also, explicit decay rates for numerical schemes have only recently been established. In \cite{Banda2013}, exponential decay on a finite time horizon has been developed for non-conservative (first-order) schemes and in~\cite{Goettlich2017} for discretizations of linear systems only.  
Different to the already existing results for linear feedback laws, we derive a mixed feedback law based on the numerical discretization, 
that is capable to tackle instances, where queues are still filled and hence the ODE affects the stability behavior.

This paper is organized as follows: in Section \ref{sec:prod}, we introduce the production model that presented in \cite{DApice2010} and provide a suitable numerical discretization. We make use of the theory of Lyapunov functions (Section \ref{sec:lyap}) to derive a feedback law that ensures the exponential stability for the coupled PDE-ODE model in Section \ref{sec:analytical}. The theoretical results are accompanied by various simulation results. 
Since we are also interested in the exponential stability of the numerical solution, we prove its exponential stability and additionally provide decay rates in Section \ref{sec:num}. On the discrete level, we obtain a so-called mixed feedback law to deal with queueing situations, i.e. non-empty queue loads.

\section{Model Equations and Numerical Discretization}
\label{sec:prod}
First, we recall the production model taken from \cite{DApice2010}. For simplicity, we stick to the special case of a serial system without any branches. The model considers a queue in front of every processing unit $e\in \{1,...,m\}$. Furthermore, the maximal capacity $\mu_e>0$ and the processing velocity $v_e>0$ of each unit are constant parameters. 
The dynamics of the product density $\rho_e(t,x)$ defined on a segment $[a_e,b_e]$ follows \cite{ArmbrusterDegondRinghofer2006aa}:
\begin{subequations}
	\begin{equation}
	\partial_t\rho_e(t,x)+\partial_xf(\rho_e(t,x))=0,
	\label{eq:21a}
	\end{equation}
	\begin{equation}
	f_e(\rho_e(t,x))=\min\{{v}_e\rho_e(t,x),\mu_e\},~~~~\forall x\in[a_e,b_e], \, t\in\mathbb{R_+},
	\end{equation}
	\begin{equation}
	\rho_e(0,x)=\rho_{e,0}(x),
	\end{equation}
\end{subequations}
where $f_e$ denotes the flux function. We define the inflow $g_{in,e}$ at $x=a_e$ in front of processing unit $e$ as the outflow of the predecessor:
\begin{equation}
g_{in,e}(t)\coloneqq f_{{e-1}}(\rho_{{e-1}}(t,b_{e-1})),~\text{for~}e=2,...,m.
\label{eq:inflow}
\end{equation}
In the case $e=1$, an externally given inflow profile $g_{in,1}(t)=f_{in}(t)$ is assumed.
Due to the possibility of different maximal capacities $\mu_e$, it might occur that the inflow 
can not be processed immediately. Therefore, we introduce a time-dependent function $q_e(t)$ describing the load of the queue at processing unit $e$.
Each queue satisfies the following ordinary differential equation:
\begin{align}
\partial_tq_e(t)=g_{in,e}(t)-g_{out,e}(t),~\text{with~}g_{out,e}(t) \coloneqq f_e(\rho_e(t,a_e)) \label{eq:queue}
\end{align}
describing the outflow from queue $e$ to processor $e$ at time $t$, cf. Figure \ref{img:ginout}. 
\begin{figure}[H]
	\centering
	\begin{tikzpicture}[
	decoration={markings,mark=at position 1 with {\arrow[scale=1,black]{latex}};}]
	\draw  (0,0) circle (3mm) node {$q_1$};
	\draw [postaction={decorate}] (0.3,0) -- (3.3,0); 
	\draw  (3.6,0) circle (3mm) node (v1) {$q_2$};
	\node (d1) [draw=none,fill=none] at (3.9,0) {~};
	\path[->] (d1) edge [loop above] node {~}(d1);
	\node (d2) [draw=none,fill=none] at (3.2,0) {~};
	\path[->] (d2) edge [loop above] node {~}(v1);
	\draw [postaction={decorate}] (3.9,0) -- (6.9,0);
	\coordinate[label=above:1] ()  at (1.8,0);
	\coordinate[label=above:2] ()  at (5.4,0);
	\coordinate[label=above:$g_{out,2}$] ()  at (4.0,0.5);
	\coordinate[label=above:$g_{in,2}$] ()  at (3.0,0.5);
	\end{tikzpicture}
	\caption{Serial production system with two processors and two queues}
	\label{img:ginout}
\end{figure}
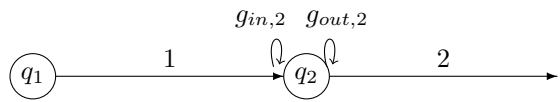

Obviously, the outgoing flux $g_{out,e}$ depends on the queue load. 
So if the queue is empty, the outflow is the minimum of the incoming flux $g_{in,e}$ and the maximum processing capacity $\mu_e$. In the first case, the queue remains empty, whereas in the second case, the queue starts to increase. In the non-empty case, the outgoing flux is equal to the maximal capacity. Summarizing, we get
\begin{equation}
g_{out,e}(t)\coloneqq f_e(\rho_e(t,a_e))=\begin{cases}
\min\{f_{{e-1}}(\rho_{{e-1}}(t,b_{e-1})),\mu_e\},&q_e(t)=0,\\
\mu_e,&q_e(t)>0.
\end{cases}
\label{eq:queuecoup}
\end{equation}
Equations \eqref{eq:queue} and \eqref{eq:queuecoup} ensure that the production model is well-defined, provided that $f_{in}(t)$ and $\rho_{e,0}(x)$ are of total bounded variation, see \cite{DApice2010}.

To avoid discontinuities in the derivative of the queue length, a smoothed out version of \eqref{eq:queuecoup} has been derived in \cite{Armbruster} and reads as
\begin{equation}
f_e(\rho_e(t,a_e))\approx\min\left\{\mu_e,\frac{q_e(t)}{\epsilon}\right\},~\text{with}~\epsilon\ll1.
\label{eq:approx}
\end{equation}
Provided ${v}_e\rho_{e,0}\leq\mu_e$, \eqref{eq:21a} reduces to a linear advection equation
\begin{equation*}
\partial_t\rho_e+v_e\partial_x\rho_e=0.
\end{equation*} 
For the numerical investigations later on, we introduce a first-order discretization of the coupled production model. 
From now on, we assume that $x\in[0,l]$, where $l$ is the uniform length of all processors. 
It is advantageous to rewrite the model in terms of fluxes only to apply the Lyapunov stability concepts in Section \ref{sec:lyap}.
By defining 
\begin{align*}
\Lambda=\text{diag}({v}_1,\ldots,{v}_m), \quad f(t,x)=(f_1(t,x),\ldots,f_m(t,x))^T , \quad \partial_tq(t)=(\partial_tq_1(t),\ldots,\partial_tq_m(t))^T,\\ 
 g_{in}(t)=(g_{in,1}(t),\ldots,g_{in,m}(t))^T \text{ and }\quad g_{out}(t)=(g_{out,1}(t),\ldots,g_{out,m}(t))^T, \hspace*{0.08\textwidth}
\end{align*}
we are able to write the model in the following compact form:
\begin{subequations}
	\begin{equation}
	\partial_tf(t,x)+\Lambda\partial_xf(t,x)=0,~x\in[0,l],\, t\in\mathbb{R_+},
	\label{eq:fcompact1}
	\end{equation}
	\begin{equation}
	\partial_tq(t)=g_{in}(t)-g_{out}(t),
	\label{eq:fcompact2}
	\end{equation}
	\label{eq:fcompact}
\end{subequations}
equipped with initial values $f_{e,0}(x)=v_e\rho_{e,0}(x),~q(0)=q_{0}\geq0$. Note that \eqref{eq:fcompact1} is already given in characteristic form, see \cite{Bastin2016}.
 
For the discretization of \eqref{eq:fcompact1} we apply the left-sided upwind method due to the strictly positive velocities $v_e$.
The equation for the queues is approximated by an explicit Euler scheme as proposed in \cite{DApice2010}. 
We use an equidistant grid with constant time step size $\tau$ and constant space step size $h$. To describe the boundary values, we use one ghost cell at the left. We consider $N+1$ cells with cell centers $x_j=(j+\frac{1}{2})h$ for $j\in\{-1,\ldots,N-1\}$. The interfaces of the cells are located at $x_{j-\frac{1}{2}}=jh$ for $j\in\{-1,...,N-1\}$. The considered time domain is $[0,T]$, whereby the discrete time is denoted by $t_{k}=k\tau$ for $k\in\{0,...,K-1\}$ such that $(K-1)\tau=T$. Moreover, the CFL condition is assumed to be fulfilled, i.e.,
\begin{equation}
\bar{v}\frac{\tau}{h}\leq 1,~~~ \bar{v}\coloneqq\max_{e\in\{1,...,m\}}{v}_e.
\label{eq:cfl1}
\end{equation}

For further investigations, we assume that 
$q_1^k=0~\forall k$, which is satisfied if $f_{in}^k\leq\mu_1~\forall k$. 
Then, the discretized equations for $j\in\{0,...,N-1\}$ and $k\in\{0,...,K-2\}$ read:
\begin{subequations}
	\begin{align}
	f_{e,j}^{k+1}&=f_{e,j}^k-\frac{\tau}{h}{v}_e\left(f_{e,j}^k-f_{e,j-1}^k\right),&\text{for}~e\in\{1,...,m\},\label{eq:d1}\\
	f_{e,-1}^k&={g_{out,e}}^k=\begin{cases}
	\min\{f_{e-1,N-1}^k,\mu_e\},~~~q_e^k=0,\\
	\mu_e,~~~~~~~~~~~~~~~~~~~~~~~~q_e^k>0,
	\end{cases}&\text{for}~e\neq1,\label{eq:d2}\\
	f_{1,-1}^k&={g_{out,1}}^k=f_{in}^k,\label{eq:d3}\\
	q_{e}^{k+1}&=q_{e}^k+\tau({g_{in,e}}^k-{g_{out,e}}^k),~\text{with}&\text{for}~e\in\{1,...,m\},\label{eq:d4}\\
	\nonumber {g_{in,e}}^k&=\begin{cases}f_{e-1,N-1}^k,~e\neq1, \\
	f_{in}^k,~~~~~~~~~\text{else},\end{cases}\\
	f_{e,j}^0& = f_{0,e},~q_{e}^0= q_{0,e},&\text{for}~e\in\{1,...,m\}.\label{eq:d5}
	\end{align}
	\label{eq:discretization}
\end{subequations} 
Here, \eqref{eq:d2},\eqref{eq:d3} are the discretized boundary conditions that are determined by the queue or the control input. 
We note that we distinguish between the source node ($e=1$) and internal nodes ($e = 2,\dots,m$). Due to assumption $f_{in}^k\leq\mu_1~\forall k$, 
it is valid to set $g_{out,1}(t)$ equal to the inflow profile $f_{in}(t)$. The discretized production network model is referred to as PNM in the following.

\section{Feedback Stabilization for the Continuous System}

We now investigate the stabilization of the continuous model \eqref{eq:fcompact}. In this section, starting with the Lyapunov stability, we discuss results for a serial network consisting of only two processors and provide numerical illustrations. We remark that the extension to more processing units is
straightforward.
\subsection{Lyapunov Stability}
\label{sec:lyap}
In the following, we analyze the exponential stability of the continuous system \eqref{eq:fcompact}. 
In contrast to \cite{Bastin2016,Bastin2007}, we are faced with queues and therefore investigate the exponential stability in the norm
\begin{equation}
\lVert(f(t,\cdot),q(t))\rVert=\sqrt{\lVert f(t,\cdot)\rVert_{L^2((0,l);\mathbb{R}^m)}^2+|q(t)|^2},
\label{eq:newnorm}
\end{equation}
where $|\cdot|$ is the Euclidean norm. The introduced norm \eqref{eq:newnorm} is the natural choice of the product space $L^2((0,l);\mathbb{R}^m)\times \mathbb{R}^m$ motivated by \cite{Krstic2008,Tang2011}.
\begin{defi}
	The system \eqref{eq:fcompact1}-\eqref{eq:fcompact2} is exponentially stable in the sense of the norm \eqref{eq:newnorm}, if there exist $\nu>0$ and $C>0$ such that for every initial condition $f_0\in L^2((0,l);\mathbb{R}^m)$ and $q(0)\in \mathbb{R}^m$ the solution to \eqref{eq:fcompact} satisfies
	\begin{equation}
	\lVert f(t,\cdot)\rVert_{L^2((0,l);\mathbb{R}^m)}^2+|q(t)|^2\leq Ce^{-\nu t}\left(\lVert f(0,\cdot)\rVert_{L^2((0,l);\mathbb{R}^m)}^2+|q(0)|^2\right),~\forall t\in\mathbb{R}_+.
	\label{eq:exponentialstability1}
	\end{equation}	
	\label{defi:exp}
\end{defi}
Note that we are particularly interested in the stabilization of the trivial steady state $f\equiv0$ and $q\equiv0$. Nevertheless, the model also allows for non-trivial states, which read as
$f \equiv z$, $q \equiv 0$, where
\begin{equation*}
z=\left(z_1, \dots,z_m\right),~\text{with}~z_e=\zeta\min_e(\mu_e)~\text{for}~0\leq\zeta<1.
\end{equation*}
By translation $\tilde{f}=f-z$, we recover results for the non-trivial steady states from the trivial steady state again.
For the feedback analysis we make use of a Lyapunov function approach. Since we consider a serial system with queues, the Lyapunov function already known for 
hyperbolic equations (see e.g. \cite{Bastin2016,Diagne2012,Goettlich2017}) needs to be adapted. We define the following {\em Lyapunov function candidate}:
\begin{align}
\nonumber V(t)&=V_1(t)+V_2(t)\\
&=\int_0^l f^T(t,x)P(x)f(t,x)~dx+q^T(t)Q(t)q(t),~
\label{eq:lyap}
\end{align} with the weighting matrices
$P(x)=\text{diag}(p_{1}e^{-\eta _1x},...,p_{m}e^{-\eta_mx})$, $p_{e}>0,\eta_e>0$ and\\$Q(t)=\text{diag}(c_{1}e^{-\tilde{\eta}_1 {v}_1t},\ldots,c_{m}e^{-\tilde{\eta}_m{v}_mt})$, $c_{e}>0,\tilde{\eta}_e>0$. In order to include the ODEs for the queues, we add the quadratic form $V_2(t)=q(t)^TQ(t)q(t)$. If not stated otherwise, we assume $\eta_e=\eta$ and $\tilde{\eta}_e=\tilde{\eta}$, independent of the individual properties. 

\begin{rem}
For the discrete stability analysis in Section \ref{sec:num}, we define the {\em discrete version of the Lyapunov function} in \eqref{eq:lyap} as
\begin{equation}
V^k=\sum\limits_{j=0}^{N-1}\sum\limits_{e=1}^{m}(f_{e,j}^k)^2p_{e}e^{-\eta_e x_{e,j}}h+\sum\limits_{e=1}^{m}(q_{e}^k)^2c_{e}e^{-\tilde{\eta}_e{v}_et_k},~~~\text{with}~\eta_e,\tilde{\eta}_e,p_e,c_e>0.
\label{eq:disclya1}
\end{equation}
\end{rem}

\subsection{Analytical Feedback Stabilization}
\label{sec:analytical}
The goal is to derive a feedback law using the Lyapunov function from \eqref{eq:lyap}. 
Figure \ref{img:ginout1} illustrates the key idea of such a closed-loop system, where
a percentage of the outflow $f_2(t,l)$ is fed back into the system.
We aim to derive a control input $u_1(t)$ that depends on the outflow of the system such that exponential stability is ensured. 
We call this control input a {\em feedback law}.
\begin{figure}[h]
	\centering
	\begin{tikzpicture}[
	decoration={markings,mark=at position 1 with {\arrow[scale=1,black]{latex}};}]
	\tikzstyle{place}=[circle,draw=black!50,fill=white!20,thick]
	\tikzstyle{transition}=[rectangle,draw=black!50,fill=black!20,thick]
	\draw  (0,0) circle (3mm) node (v0) {$q_1$};
	\draw [postaction={decorate}] (0.3,0) -- (3.3,0); 
	\draw  (3.6,0) circle (3mm) node (v1) {$q_2$};
	\node (d1) [draw=none,fill=none] at (3.9,0) {~};
	\draw [postaction={decorate}] (3.9,0) -- (6.9,0);
	\coordinate[label=above:1] ()  at (1.8,0);
	\coordinate[label=above:2] ()  at (5.4,0);
	\node (d4) [draw=none,fill=none] at (-0.3,-0.2) {~};
	\node (d3) [draw=none,fill=none] at (7.0,0) {~};
	\draw [arrow,dashed,bend angle=25,bend left]  (d3) to (d4);
	\node (d2) [draw=none,fill=none] at (7.0,0.3) {$f_2(t,l)$};
	\node (d3) [draw=none,fill=none] at (3.5,-1.4) {$u_1(t)$};
	\node (d4) [draw=none,fill=none] at (-0.7,-0.5) {$g_{in,1}(t)$};
	\end{tikzpicture}
	\caption{Feedback loop with control input $u_1(t)$}
	\label{img:ginout1}
\end{figure}
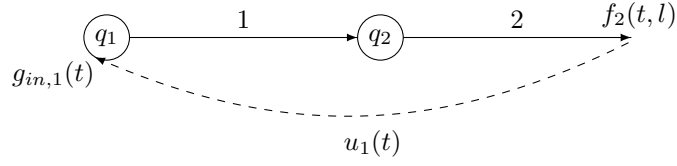
\label{ex:ex2}

\begin{rem}
By considering a control input $u_1(t)$ as an inflow into the system, equations (\ref{eq:d3}) and (\ref{eq:d4}) of the numerical discretization are adapted as follows:
\begin{subequations}
\begin{align}
f_{1,-1}^k&={g_{out,1}}^k=u_{1}^k,\label{eq:d31}\\
q_{e}^{k+1}&=q_{e}^k+\tau({g_{in,e}}^k-{g_{out,e}}^k),~\text{with}&\text{for}~e\in\{1,...,m\},\label{eq:d41}\\
\nonumber {g_{in,e}}^k&=\begin{cases}f_{e-1,N-1}^k,~e\neq1,\\ 
u_{1}^k,~~~~~~~~~~\text{else}.\end{cases}
\end{align}
\end{subequations}
\end{rem}

Since so far, only the non-queue case has been studied in the literature \cite{Bastin2016,Bastin,Bastin2007}. However, 
difficulties arise when queues start to grow up. Therefore, we need to adapt the already established results accordingly.
The following lemma presents a condition that ensures the exponential decay of the Lyapunov function \eqref{eq:lyap}.
\begin{lem}
	The continuous solution satisfies
	\begin{equation*}
	\dot{V}\leq-\nu V,
	\end{equation*}
	i.e., exponential decay of $V$ is obtained, if 
	\begin{equation}
	-f(t,l)^T\Lambda P(l)f(t,l)+f(t,0)^T\Lambda P(0)f(t,0)+2\partial_tq(t)^TQ(t)q(t)\leq0
	\label{eq:condlemma}
	\end{equation}
	is fulfilled $\forall t \in\mathbb{R_+}$. 
	The decay rate $\nu$ is given by
	\begin{equation}
	\nu=\min(\eta,\tilde{\eta})\min_e({v}_e)>0.
	\end{equation}
	\label{lem:conditionu}
\end{lem}
\begin{proof}
	In the following, we assume that $f$ and $q$ are sufficiently smooth. We obtain:
	\begin{equation}
	\begin{array}[b]{r@{}c@{}l}
	\nonumber\dot{V} & \overset{\eqref{eq:lyap}}{=} &\int_{0}^{l}\left(\partial_tf^TP(x)f+f^TP(x)\partial_tf\right)dx+2\partial_tq^TQ(t)q+q^T\partial_tQ(t)q\\ [1ex]
	\nonumber&\overset{\eqref{eq:fcompact1}}{=} & \int_{0}^{l}\left(-\partial_xf^T\Lambda P(x)f-f^TP(x)\Lambda\partial_xf\right)dx+2\partial_tq^TQ(t)q-\tilde{\eta}q^T\Lambda Q(t)q\\ [1ex]
	\nonumber&= & -\int_{0}^{l}\partial_x\left[f^T\Lambda P(x)f\right]dx-\eta\int_{0}^{l}f^T\Lambda P(x)fdx+2\partial_tq^TQ(t)q-\tilde{\eta}q^T\Lambda Q(t)q\\ [1ex]
	\nonumber&\overset{}{\leq}&-\left[f^T\Lambda P(x)f\right]_0^l-\eta\min_e({v}_e)\int_{0}^{l}f^T P(x)fdx+2\partial_tq^TQ(t)q-\tilde{\eta}\min_e({v}_e)q^TQ(t)q\\ [1ex]
	&\leq&-\min(\eta,\tilde{\eta})\min_e({v}_e)V(t)-\left[f^T\Lambda P(x)f\right]_0^l+2\partial_tq^TQ(t)q.
	\end{array}
	\end{equation}
	To obtain the exponential decay of $V$, we deduce the condition
	\begin{equation}
	-\left[f^T\Lambda P(x)f\right]_0^l+2\partial_tq^TQ(t)q\leq0. \qedhere
	\label{eq:conditionu}
	\end{equation}
\end{proof}


Next, we derive a feedback law by using Lemma \ref{lem:conditionu} and by imposing the feedback control $u_1(t)$ at processor 1.
For the serial network in Figure \ref{img:ginout1} we have to show
\begin{align*}
& -{v}_1f_1(t,l)^2p_1e^{-\eta l}+{v}_1f_1(t,0)^2p_1e^{-\eta0}-{v}_2f_2(t,l)^2p_2e^{-\eta l}+{v}_2f_2(t,0)^2p_2e^{-\eta 0}\\[1ex]
&+2q_1(t)(g_{in,1}(t)-g_{out,1}(t))c_1e^{-\tilde{\eta}{v}_1t}+2q_2(t)(g_{in,2}(t)-g_{out,2}(t))c_2e^{-\tilde{\eta}{v}_2t}\leq0.
\end{align*}
Due to the assumption $q_1(t)=0$, the outflow into processor 1 is equal to the control input $u_1(t)$, i.e., $f_1(t,0)=u_1(t)$. This leads to
\begin{align*}
{v}_1u_1(t)^2p_1&\leq {v}_1f_1(t,l)^2p_1e^{-\eta l}+{v}_2f_2(t,l)^2p_2e^{-\eta l}\\
& -{v}_2f_2(t,0)^2p_2-2q_2(t)g_{in,2}(t)c_2e^{-\tilde{\eta}{v}_2t}+2q_2(t)g_{out,2}(t)c_2e^{-\tilde{\eta}{v}_2t}.
\end{align*}
In the next step, we set $p_1=p_2=c_1=c_2=1$ and use the approximate version of $g_{out,2}(t)=f_2(t,0)$ in \eqref{eq:approx} for $\epsilon\ll1$ to ensure enough regularity:
\begin{align}
\nonumber u_1(t)^2&\leq f_1(t,l)^2e^{-\eta l}+\frac{{v}_2}{{v}_1}f_2(t,l)^2e^{-\eta l}-\frac{{v}_2}{{v}_1}\min\left\{\frac{q_2(t)}{\epsilon},\mu_2\right\}^2\\
&-2\frac{1}{{v}_1}q_2(t)f_1(t,l)e^{-\tilde{\eta}{v}_2t}+2\frac{1}{{v}_1}q_2(t)\min\left\{\frac{q_2(t)}{\epsilon},\mu_2\right\}e^{-\tilde{\eta}{v}_2t}\eqqcolon X(t).
\label{eq:quadrat}
\end{align}
Thus, the control input $u_1(t)$ is bounded from above by $X(t)$. This is only valid if $X(t)$ is non-negative. In the steady state $f\equiv0$ and $q\equiv0$, the control $u_1=0$ is obtained.
Assuming ${v}_e={v}$ for $e\in\{1,2\}$, we get:
\begin{align}
\nonumber X(t)&= f_1(t,l)^2e^{-\eta l}+f_2(t,l)^2e^{-\eta l}-\min\left\{\frac{q_2(t)}{\epsilon},\mu_2\right\}^2\\
&-2\frac{1}{{v}}q_2(t)f_1(t,l)e^{-\tilde{\eta}{v}t}+2\frac{1}{{v}}q_2(t)\min\left\{\frac{q_2(t)}{\epsilon},\mu_2\right\}e^{-\tilde{\eta}{v}t}.
\label{eq:X_t}
\end{align}
Assuming ${v}_e={v}$ and $p_e=p=c_e=c=1$, this result can be extended to a serial system with $m$ processors:  
\begin{align}
\nonumber X(t)&=\sum\limits_{e=1}^{m}f_e(t,l)^2e^{-\eta l}-\sum\limits_{e=2}^{m}\min\left\{\frac{q_e(t)}{\epsilon},\mu_e\right\}^2-\sum\limits_{e=2}^{m}2\frac{1}{{v}}q_e(t)f_{e-1}(t,l)e^{-\tilde{\eta}{v}t}\\
&+\sum\limits_{e=2}^{m}2\frac{1}{{v}}q_e(t)\min\left\{\frac{q_e(t)}{\epsilon},\mu_e\right\}e^{-\tilde{\eta}{v}t}.
\label{eq:Xgen}
\end{align}

If $X(t)<0$, we may set $u_1(t)=0$. Otherwise, we set $u_1(t)=\sqrt{X(t)}$ motivated by \eqref{eq:quadrat}. We observe that $u_1(t)$ depends on the values of $f$ at the boundaries and on the queue load. 
This leads to a linear feedback control, i.e., a percentage of the outflow that is fed back into the system,
\begin{equation}u_{1}(t) = \kappa f_m(t,l),\label{eq:kappa1}
\end{equation}
for $\kappa>0$. 	

In the case of vanishing queues, i.e.\ $q_e(t)=0$ for all $t\in\mathbb{R}_+$, the results from Lemma \ref{lem:conditionu} and \eqref{eq:quadrat}--\eqref{eq:kappa1} coincide with \cite{Bastin2016,Bastin2007,Goettlich2017} by a direct comparison.  

\subsection{Simulation Results}
In the following, we provide simulation results 
for the linear feedback law for queueing situations. 
If not stated otherwise, we use a time horizon of $T=50$ and $h=\tau=0.01$ for the step sizes.
All queues shall be empty at $t=0$, i.e. $q_e(0)=0$.
We set ${v}_e=1$ and the length $l$ of each processor is 1.

We consider the linear feedback control law \eqref{eq:kappa1} for three processors and two active queues. 
Our first approach to generate queues is to set the initial condition for the fluxes equal to the maximal capacity.
We choose $\mu_1=10,~\mu_2=9,~\mu_3=8$ and the initial conditions $f_1(0,x)=10$,~$f_2(0,x)=9$,~$f_3(0,x)=8$. The latter choice leads to queues at the second and third processor. The parameter for the weighting matrices of the Lyapunov function are $\eta_e=\tilde{\eta}_e=0.1$ and $p_e=c_e=1$. 
We denote the upper bound from the proof of Lemma \ref{lem:conditionu} by:
\begin{equation}
V_{up}(t)\coloneqq V(0)\exp(-\min(\eta,\tilde{\eta})\min_e(v_e)t).
\label{eq:Vup}
\end{equation}
Figure \ref{img:lyapcompl} shows that the estimated Lyapunov function for $\kappa=0.1$ has a kink and is even slightly increasing for a short time interval because of the queues in front of processor 2 and 3. 
The Lyapunov function for the queue ($V_2$) is illustrated in Figure \ref{img:lyapcomp}. 

\begin{figure}[htb!]
	\centering
	\begin{subfigure}[b]{.49\textwidth}
		\centering
		\includegraphics[width=\textwidth]{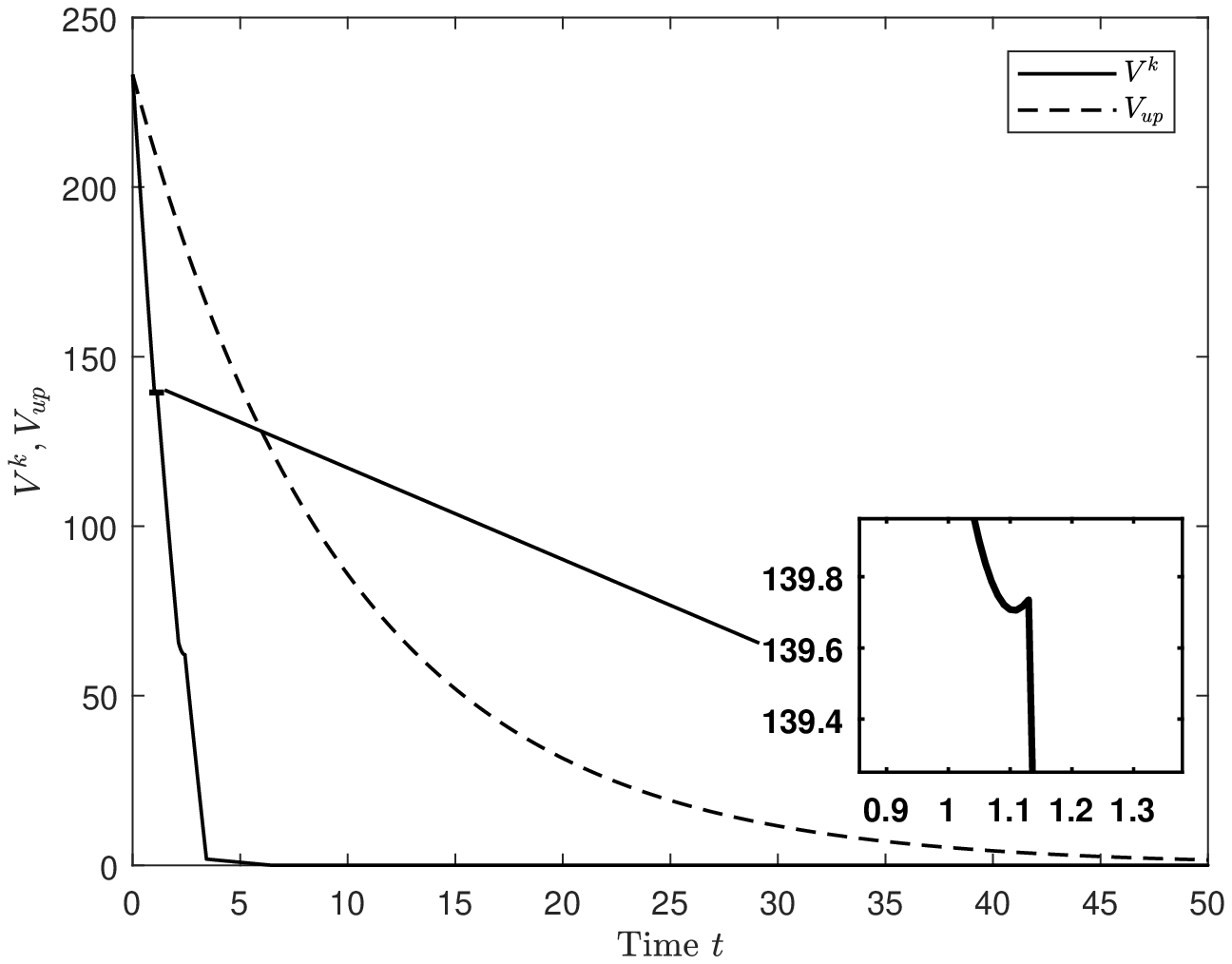}
		\caption{Discrete Lyapunov function $V^k$ \eqref{eq:disclya1} and upper bound $V_{up}$ \eqref{eq:Vup}}
		\label{img:lyapcompl}
	\end{subfigure}
	\begin{subfigure}[b]{.49\textwidth}
		\includegraphics[width=\textwidth]{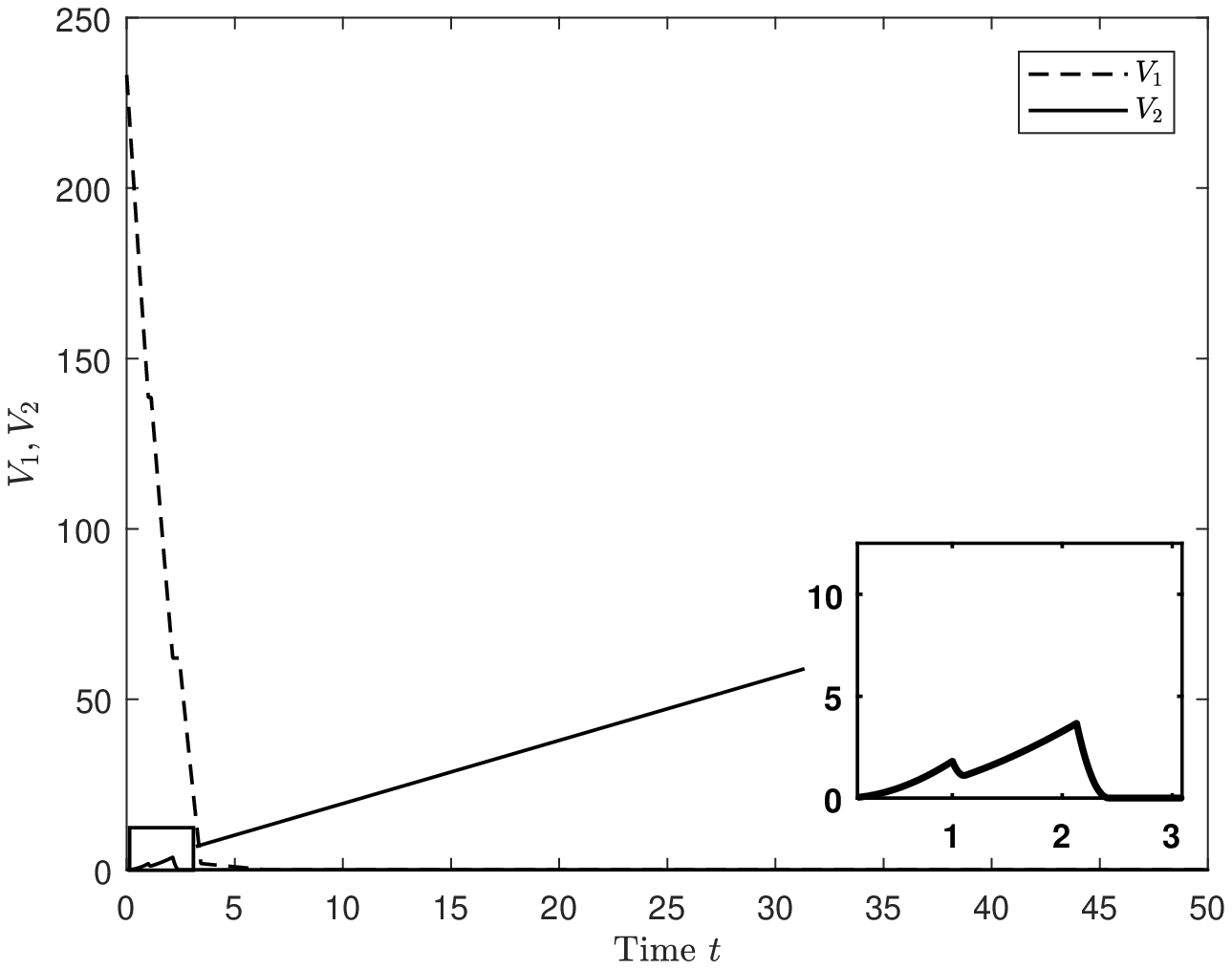}	
		\caption{Lyapunov function for the fluxes $V_1$ and for the queues $V_2$, see \eqref{eq:lyap}}
		\label{img:lyapcomp}
	\end{subfigure}
	\caption{Discrete Lyapunov function}
\end{figure}

Zooming in, we observe that this Lyapunov function decreases before it starts to increase again. This is due to the chosen initial conditions that cause queue 3 to continue growing while queue 2 already shrinks. Even though the Lyapunov function for the flux is decreasing or temporarily constant, the increase in queue 3 can not be compensated. Furthermore, we can see that the Lyapunov function for the flux is temporarily constant during two time intervals which is caused by the existence of two non-empty queues. 
Summarizing, there exist scenarios using the linear feedback law, where we are not able to ensure asymptotic stability, i.e.\ a decreasing Lyapunov function, for all time steps. However, such kinks can be avoided applying a so-called mixed feedback law, see Section \ref{sec:num}.

\section{Feedback Stabilization for the Discretized Model}
\label{sec:num}
In this section, we study how the results from our theoretical investigations can be transferred to the
discretized equations, similar to \cite{Banda2013,Goettlich2017}.
We present decay rates for the exponential stability and derive a new type of feedback law
that is able to avoid increasing Lyapunov functions in case of non-empty queues.

Considering a discrete Lyapunov function \eqref{eq:disclya1} also requires the definition of a discrete norm. 
Therefore, we use for the stability analysis of the disctrized equations
\begin{equation}
\lVert(f,q)\rVert_h=\sqrt{\lVert f\rVert_{h,2}^2+|q|^2},~~~~~(f,q)\in\mathbb{R}^{m\times N}\times\mathbb{R}^m,
\label{eq:newnorm1}
\end{equation}
where $\lVert f\rVert^2_{h,2}=\sum\limits_{j=0}^{N-1}\sum\limits_{e=1}^{m}(f_{e,j})^2h$ is the discrete $L^2$-norm of $f$.

Since we are interested in a feedback stabilization for the numerical discretization, we equip \eqref{eq:fcompact} with
\begin{subequations}
	\begin{align}
	&g_{in}(t)=Gf(t,l),
	\label{eq:fcompact31}
	\end{align}
\end{subequations}
where the matrix $G\in\mathbb{R}^{m\times m}$ is the inflow matrix with constant entries, coupling the inflow with the outflow of the system. 
The following theorem states under which conditions the numerical discretization for the serial network model is exponentially stable.
\begin{theo}
	\label{theo:discr}
	Let $T>0$ be arbitrarily large but fixed, $f_{in}^k \leq \mu_1$ 
	and the feedback law $f_{in}^k=u_1^k=f_{1,-1}^k$ satisfies for $k\in\{0,...,K-1\}$ the inequality:  
	\begin{align} 
	\nonumber  &\sum\limits_{e=1}^{m}{v}_e\left((f_{e,-1}^k)^2p_e\exp(-\eta_e x_{e,0})-(f_{e,N-1}^k)^2p_e\exp(-\eta_ex_{e,N-1})\right)\\
	&+\sum\limits_{e=1}^{m}\left(2q_{e}^k(g_{in,e}^k-g_{out,e}^k)+\tau(g_{in,e}^k-g_{out,e}^k)^2\right)c_e\exp(-\tilde{\eta}_e{v}_et_{k})\exp(-\tilde{\eta}_e{v}_e\tau)\leq0. \label{ineq}
	\end{align}
	Then, provided the CFL condition \eqref{eq:cfl1} is fulfilled, the following holds: there exist $\eta_e>0$ and $\tilde{\eta}_e>0$ such that the numerical scheme \eqref{eq:d1}-\eqref{eq:d5} is exponentially stable, i.e., the numerical solution satisfies
	\begin{equation}
	V^{k+1}\leq \exp(-\nu\tau(k+1))V^{0},~\text{for}~k\in\{0,\ldots,K-2\},
	\label{eq:Vup1}
	\end{equation}
	where the decay rate $\nu$ is given by:
	\begin{equation}\nu=\min\Big\{\frac{1}{h}\min_e({v}_e(1-\exp(-\eta_eh))),\frac{1}{h}\max_e({v}_e)\min_e(1-\exp(-\tilde{\eta}_e{v}_e\tau))\Big\}>0.
	\label{eq:nu0}
	\end{equation} 
	Furthermore, $f_{e,j}^{k+1}$ and $q_e^{k+1}$ are exponentially stable in the sense of the discrete norm in \eqref{eq:newnorm1}, i.e., 
	\begin{equation}
	\sum\limits_{j=0}^{N-1}\sum\limits_{e=1}^{m}(f_{e,j}^{k+1})^2h+\sum\limits_{e=1}^{m}(q_{e}^{k+1})^2\leq C\exp(-\nu\tau(k+1))\left(\sum\limits_{j=0}^{N-1}\sum\limits_{e=1}^{m}(f_{e,j}^0)^2h+\sum\limits_{e=1}^{m}(q_{e}^0)^2\right),
	\label{eq:normabsch}
	\end{equation}
	for $k\in\{0,...,K-2\}$ and $C>0$.
\end{theo}
 
\begin{proof}
	In order to obtain the exponential stability of the numerical solution, we show that the discrete time derivative of the discrete Lyapunov function in \eqref{eq:disclya1} fulfills:
	\begin{equation*}
	\frac{V^{k+1}-V^k}{\tau}\leq-\nu V^k,~\text{for a decay rate}~\nu>0.
	\end{equation*}
	By inserting the discrete Lyapunov function at time $t^k$ and $t^{k+1}$, we obtain:
	
	\begin{align*}
	\frac{V^{k+1}-{V}^k}{\tau}&=\sum\limits_{j=0}^{N-1}\sum\limits_{e=1}^{m}\left(\left(f_{e,j}^{k+1}\right)^2-\left(f_{e,j}^k\right)^2\right)p_e\exp(-\eta_ex_{e,j})\frac{h}{\tau}\\
	&+\sum\limits_{e=1}^{m}\left((q_e^{k+1})^2c_e\exp(-\tilde{\eta}_e{v}_et_{k+1})-(q_e^k)^2c_e\exp(-\tilde{\eta}_e{v}_et_k)\right)\frac{1}{\tau}\\
	&=C_1+C_2,
	\end{align*}
	with
	\begin{align*}
	C_1&=\sum\limits_{j=0}^{N-1}\sum\limits_{e=1}^{m}\left(\left(f_{e,j}^{k+1}\right)^2-\left(f_{e,j}^k\right)^2\right)p_e\exp(-\eta_ex_{e,j})\frac{h}{\tau},\\
	C_2&=\sum\limits_{e=1}^{m}\left((q_e^{k+1})^2c_e\exp(-\tilde{\eta}_e{v}_et_{k+1})-(q_e^k)^2c_e\exp(-\tilde{\eta}_e{v}_et_k)\right)\frac{1}{\tau}.
	\end{align*}
	At first, we have a closer look at $C_1$ by using the numerical discretization scheme \eqref{eq:d1}-\eqref{eq:d3} for the advection equation:
	\begin{equation*}
	\begin{array}[b]{r@{}c@{}l}
	C_1 & \overset{\eqref{eq:d1}}{=} & \sum\limits_{j=0}^{N-1}\sum\limits_{e=1}^{m}\left(\left(f_{e,j}^k-\frac{\tau}{h}v_e(f_{e,j}^k-f_{e,j-1}^k)\right)^2-\left(f_{e,j}^k\right)^2\right)p_e\exp(-\eta_ex_{e,j})\frac{h}{\tau}\\
	& = & \sum\limits_{j=0}^{N-1}\sum\limits_{e=1}^{m}\left(\left(1-\frac{\tau}{h}v_e\right)f_{e,j}^k+\frac{\tau}{h}v_ef_{e,j-1}^k\right)^2p_e\exp(-\eta_ex_{e,j})\frac{h}{\tau}\\
	& - & \sum\limits_{j=0}^{N-1}\sum\limits_{e=1}^{m}\left(f_{e,j}^k\right)^2p_e\exp(-\eta_ex_{e,j})\frac{h}{\tau}.
	\end{array}
	\end{equation*}
	Applying the CFL condition, 
	we obtain
	\begin{align}
	C_1 &\leq\sum\limits_{j=0}^{N-1}\sum\limits_{e=1}^{m}\left(\left(1-\frac{\tau}{h}v_e\right)\left(f_{e,j}^k\right)^2+\frac{\tau}{h}v_e\left(f_{e,j-1}^k\right)^2\right)p_e\exp(-\eta_ex_{e,j})\frac{h}{\tau} \nonumber \\
	& - \sum\limits_{j=0}^{N-1}\sum\limits_{e=1}^{m}\left(f_{e,j}^k\right)^2p_e\exp(-\eta_ex_{e,j})\frac{h}{\tau} \nonumber\\
	& = \sum\limits_{j=0}^{N-1}\sum\limits_{e=1}^{m}\left(\frac{\tau}{h}v_e\left(f_{e,j-1}^k\right)^2-\frac{\tau}{h}v_e\left(f_{e,j}^k\right)^2\right)p_e\exp(-\eta_ex_{e,j})\frac{h}{\tau} \nonumber \\
	& = \sum\limits_{e=1}^{m}v_e\sum\limits_{j=0}^{N-1}\left(\left(f_{e,j-1}^k\right)^2-\left(f_{e,j}^k\right)^2\right)p_e\exp(-\eta_ex_{e,j}). \nonumber
	\end{align}
	In the next step, we make use of an index shift that requires the knowledge of the cell center $x_N$ outside of the domain which is not needed in our discretization, since we only consider positive velocities. Therefore, by assuming $x_{e,N}=x_{e,N-1}$, we get:
	\begin{equation*}
	\begin{array}[b]{r@{}c@{}l}
	C_1 & \leq & \sum\limits_{e=1}^{m}v_e\left(\sum\limits_{j=-1}^{N-2}\left(f_{e,j}^k\right)^2p_e\exp(-\eta_ex_{e,j+1})-\sum\limits_{j=0}^{N-1}\left(f_{e,j}^k\right)^2p_e\exp(-\eta_ex_{e,j})\right)\\
	& = & \sum\limits_{e=1}^{m}v_e\left(\sum\limits_{j=0}^{N-1}\left(f_{e,j}^k\right)^2\left(p_e\exp(-\eta_ex_{e,j+1})-p_e\exp(-\eta_ex_{e,j})\right)\right)\\
	& + & \sum\limits_{e=1}^{m}v_e\left(\left(f_{e,-1}^k\right)^2p_e\exp(-\eta_ex_{e,0})-\left(f_{e,N-1}^k\right)^2p_e\exp(-\eta_ex_{e,N})\right)\\
	& \overset{x_{e,N}=x_{e,N-1}}{=} & \sum\limits_{e=1}^{m}v_e\left(\sum\limits_{j=0}^{N-1}\left(\exp(-\eta_eh)-1\right)\left(f_{e,j}^k\right)^2\left(p_e\exp(-\eta_ex_{e,j})\right)\right)\\
	& + & 
	\sum\limits_{e=1}^{m}v_e\left(\left(f_{e,-1}^k\right)^2p_e\exp(-\eta_ex_{e,0})-\left(f_{e,N-1}^k\right)^2p_e\exp(-\eta_ex_{e,N-1})\right).
	\end{array}
	\end{equation*}
	Therefore, we have
	\begin{equation*}
	C_1\leq S_1+S_2,
	\end{equation*}
	with
	\begin{align*}
	S_1&=\sum\limits_{e=1}^{m}{v}_e\left(\sum\limits_{j=0}^{N-1}\left(\exp(-\eta_eh)-1\right)(f_{e,j}^k)^2p_e\exp(-\eta_ex_{e,j})\right),\\
	S_2&=\sum\limits_{e=1}^{m}{v}_e\left((f_{e,-1}^k)^2p_e\exp(-\eta_e x_{e,0})-(f_{e,N-1}^k)^2p_e\exp(-\eta_ex_{e,N-1})\right).
	\end{align*}
Regarding $C_2$, using the numerical discretization, or more specifically, the explicit Euler method for the queues, we obtain:
	\begin{equation*}
	\begin{array}[b]{r@{}c@{}l}
	C_2&\overset{\eqref{eq:d4}}{=}& \sum\limits_{e=1}^{m}\left((q_e^{k}+\tau(g_{in,e}^k-g_{out,e}^k))^2c_e\exp(-\tilde{\eta}_e{v}_et_{k+1})-(q_e^k)^2c_e\exp(-\tilde{\eta}_e{v}_et_k)\right)\frac{1}{\tau}.\\
	\end{array}
	\end{equation*}
	In the next step, we use the binomial formula, which yields:
	\begin{equation*}
	\begin{array}[b]{r@{}c@{}l}
	C_2&=&\sum\limits_{e=1}^{m}\left((q_e^{k})^2c_e\exp(-\tilde{\eta}_e{v}_et_{k})\exp(-\tilde{\eta}_e{v}_e\tau)-(q_e^{k})^2c_e\exp(-\tilde{\eta}_e{v}_et_{k})\right)\frac{1}{\tau}\\
	&+&\sum\limits_{e=1}^{m}\left(2q_{e}^k\tau(g_{in,e}^k-g_{out,e}^k)+\tau^2(g_{in,e}^k-g_{out,e}^k)^2\right)\frac{1}{\tau}c_e\exp(-\tilde{\eta}_e{v}_et_{k})\exp(-\tilde{\eta}_e{v}_e\tau)\\
	&=&Z_1+Z_2,
	\end{array}
	\end{equation*}
	with
	\begin{align*}
	Z_1&=\sum\limits_{e=1}^{m}(q_e^{k})^2c_e\exp(-\tilde{\eta}_e{v}_et_{k})(\exp(-\tilde{\eta}_e{v}_e\tau)-1)\frac{1}{\tau},\\
	Z_2&=\sum\limits_{e=1}^{m}\left(2q_{e}^k(g_{in,e}^k-g_{out,e}^k)+\tau(g_{in,e}^k-g_{out,e}^k)^2\right)c_e\exp(-\tilde{\eta}_e{v}_et_{k})\exp(-\tilde{\eta}_e{v}_e\tau).
	\end{align*}
	
	Using the CFL condition $\max_e({v}_e)\frac{\tau}{h}\leq1$ leads to:
	\begin{equation*}
	\begin{array}[b]{r@{}c@{}l}
	S_1+Z_1&=&\sum\limits_{e=1}^{m}{v}_e\left(\sum\limits_{j=0}^{N-1}\left(\exp(-\eta_eh)-1\right)(f_{e,j}^k)^2p_e\exp(-\eta_ex_{e,j})\right)\\
	&+&\sum\limits_{e=1}^{m}(q_e^{k})^2c_e\exp(-\tilde{\eta}_e{v}_et_{k})(\exp(-\tilde{\eta}_e{v}_e\tau)-1)\frac{1}{\tau}\\
	&\overset{CFL}{\leq}&\sum\limits_{e=1}^{m}{v}_e\left(\exp(-\eta_eh)-1\right)\left(\sum\limits_{j=0}^{N-1}(f_{e,j}^k)^2p_e\exp(-\eta_ex_{e,j})\right)\\
	&+&\sum\limits_{e=1}^{m}(q_e^{k})^2c_e\exp(-\tilde{\eta}_e{v}_et_{k})(\exp(-\tilde{\eta}_e{v}_e\tau)-1)\frac{\max_e({v}_e)}{h}\\
	&\leq&\max_e({v}_e(\exp(-\eta_eh)-1))\sum\limits_{e=1}^{m}\sum\limits_{j=0}^{N-1}(f_{e,j}^k)^2p_e\exp(-\eta_ex_{e,j})\\
	&+&\max_e({v}_e)\max_e(\exp(-\tilde{\eta}_e{v}_e\tau)-1)\sum\limits_{e=1}^{m}(q_e^{k})^2c_e\exp(-\tilde{\eta}_e{v}_et_{k})\frac{1}{h}\\
	&\leq&\frac{1}{h}\max\Big\{\max_e({v}_e(\exp(-\eta_eh)-1)),\max_e({v}_e)\max_e(\exp(-\tilde{\eta}_e{v}_e\tau)-1)\Big\}V^k\\
	&=&-\nu V^k.
	\end{array}
	\end{equation*}
	Therefore, if the inequality assumption \eqref{ineq} is fulfilled, i.e., $S_2+Z_2\leq 0$, we obtain:
	\begin{align*}
	\frac{V^{k+1}-{V}^k}{\tau}&\leq S_1+Z_1\\
	&\leq-\nu V^k.
	\end{align*}
	Hence, we have $V^{k+1}-V^k\leq-\tau\nu V^k$ with $\nu=\min\{\nu_1,\nu_2\}>0$, whereby
	\begin{align}
	\nu_1&=-\frac{1}{h}\max_e({v}_e(\exp(-\eta_eh)-1))=\frac{1}{h}\min_e({v}_e(1-\exp(-\eta_eh)))~\text{and} \label{eq:nu1}\\
	\nu_2&=-\frac{1}{h}\max_e({v}_e)\max_e(\exp(-\tilde{\eta}_e{v}_e\tau)-1)=\frac{1}{h}\max_e({v}_e)\min_e(1-\exp(-\tilde{\eta}_e{v}_e\tau)).\label{eq:nu2}
	\end{align}
	Using the CFL condition $\max_e({v}_e)\frac{\tau}{h}\leq1$, implies
	\begin{equation*}
	\begin{array}[b]{r@{}c@{}l}
	0&<&\tau\nu=\frac{\tau}{h}\min\{\min_e({v}_e(1-\exp(-\eta_eh))),\max_e({v}_e)\min_e(1-\exp(-\tilde{\eta}_e{v}_e\tau))\}\\
	&\overset{\text{CFL}}{\leq}&\frac{1}{\max_e({v}_e)}\min\{\min_e({v}_e(1-\exp(-\eta_eh))),\max_e({v}_e)\min_e(1-\exp(-\tilde{\eta}_e{v}_e\tau))\}\leq 1.
	\end{array}
	\end{equation*}
	Finally, we prove estimate \eqref{eq:Vup1} by recursively applying $V^{k+1}\leq(1-\tau\nu)V^k$, which reads as
	\begin{align}
	\nonumber V^{k+1} &\leq(1-\tau\nu)V^k\leq(1-\tau\nu)^{k+1}V^0\\
	&=\exp((k+1)\ln(1-\tau\nu))V^0\leq\exp(-\nu\tau(k+1))V^0,
	\label{eq:Vab}
	\end{align}
	where we use that $\ln(1-\tau\nu)\leq-\tau\nu$ for $\tau\nu\in[0,1]$.
	
	It remains to show \eqref{eq:normabsch}. It is well-known that the discrete Lyapunov function is bounded by the discrete norm in \eqref{eq:newnorm1}:
	\begin{equation*}
	\bar{C}\left(\sum\limits_{j=0}^{N-1}\sum\limits_{e=1}^{m}(f_{e,j}^k)^2h+\sum\limits_{e=1}^{m}(q_{e}^k)^2\right)\leq V^k\leq\tilde{C}\left(\sum\limits_{j=0}^{N-1}\sum\limits_{e=1}^{m}(f_{e,j}^k)^2h+\sum\limits_{e=1}^{m}(q_{e}^k)^2\right),
	\end{equation*}
	for some $\bar{C},\tilde{C}>0$.
	The above norm estimation in combination with \eqref{eq:Vab} leads to the desired result:
	\begin{equation*}
	\sum\limits_{j=0}^{N-1}\sum\limits_{e=1}^{m}(f_{e,j}^{k+1})^2h+\sum\limits_{e=1}^{m}(q_{e}^{k+1})^2\leq C\exp(-\nu\tau(k+1))\left(\sum\limits_{j=0}^{N-1}\sum\limits_{e=1}^{m}(f_{e,j}^0)^2h+\sum\limits_{e=1}^{m}(q_{e}^0)^2\right),
	\end{equation*}
	for $k\in\{0,...,K-2\}$ and $C=\frac{\tilde{C}}{\bar{C}}>0$.
\end{proof}

The inequality assumption \eqref{ineq} in Theorem \ref{theo:discr} is a crucial point for exponential stability of the discretized problem. Therefore, we focus on this assumption in the following and start with the subsequent Remark \ref{rem:rem1}.

\begin{rem}
	For sufficiently large $k$, assumption \eqref{ineq} reduces to the condition that the matrix $G^TP_0\Lambda G-P_{N-1}\Lambda$ is negative semidefinite, where $G$ is the inflow matrix from \eqref{eq:fcompact31} and $P_j$ is defined as $P_j=\text{diag}(p_1\exp(-\eta_1x_{1,j}),\ldots,p_m\exp(-\eta_mx_{m,j}))$. 
	\label{rem:rem1}
\end{rem}
	
	Let us consider again a serial system with two processing units and queues.
	The first intuition is then to insert the coupling conditions for $f_{e,-1}^k$. However, due to the queues determining the coupling conditions, this is 
	comparable to the non-queue case, cf. \cite{Goettlich2017}. We can expect the queue at processor 2 to be equal to zero for $t\rightarrow T$, i.e., $q_2^k=0$ for $k\rightarrow(K-1)$, provided that $T$ is large enough. Therefore, we differentiate two cases, i.e., a queue and no queue at processor 2. 
	\begin{enumerate}
		\item We obtain the non-queue case for $k\rightarrow(K-1)$ leading to the expression $S_2+Z_2$:
		\begin{equation*}
		\begin{array}[b]{r@{}c@{}l}
		S_2+Z_2&\overset{Z_2=0}{=}&\sum\limits_{e=1}^{2}{v}_e\left((f_{e,-1}^k)^2p_e\exp(-\eta_ex_{e,0})-(f_{e,N-1}^k)^2p_e\exp(-\eta_ex_{e,N-1})\right)\\
		&=&\sum\limits_{e=1}^{2}{v}_e\left(\Big(\sum\limits_{l=1}^{2}G_{e,l}f_{l,N-1}^k\Big)^2p_e\exp(-\eta_ex_{e,0})\right)\\
		&-&\sum\limits_{e=1}^{2}{v}_e\left((f_{e,N-1}^k)^2p_e\exp(-\eta_ex_{e,N-1})\right)\\
		&=&(f^k)^TG^TP_0\Lambda G(f^k)-(f^k)^TP_{N-1}\Lambda(f^k),
		\end{array}
		\end{equation*}
		with $f^k=(f_{1,N-1}^k,f_{2,N-1}^k)^T$. To satisfy assumption \eqref{ineq}, we have to ensure that the matrix $G^TP_0\Lambda G-P_{N-1}\Lambda$ is negative semidefinite. Then, by assuming 
		\begin{equation}\label{matrixG}
		G = \begin{pmatrix}
		0 & \kappa \\
		1 & 0 
		\end{pmatrix},
		\end{equation}
		we are able to determine $\kappa>0$ such that a linear feedback control can be applied.
		\item For small $k$ we might have $q_2^k>0$ and $f_{2,-1}^k=\mu_2$. Therefore, we are not able to find $\kappa$ such that exponential stability is obtained. In fact, we have:
		\begin{equation*}
		\begin{array}[b]{r@{}c@{}l}
		S_2&=&{v}_1\big((f_{1,-1}^k)^2p_1\exp(-\eta_1 x_{1,0})-(f_{1,N-1}^k)^2p_1\exp(-\eta_1x_{1,N-1})\big)\\ [1ex]
		&+&{v}_2\big((\mu_2)^2p_2\exp(-\eta_2 x_{2,0})-(f_{2,N-1}^k)^2p_2\exp(-\eta_2x_{2,N-1})\big),\\ [1ex]
		Z_2&\overset{q_1^k=0}{=}&2q_2^k\big(f_{1,N-1}^k-\mu_2\big)c_2\exp(-\tilde{\eta}_2{v}_2t_k)\exp(-\tilde{\eta}_2{v}_2\tau)\\ [1ex]
		&+&\tau\big(f_{1,N-1}^k-\mu_2\big)^2c_2\exp(-\tilde{\eta}_2{v}_2t_k)\exp(-\tilde{\eta}_2{v}_2\tau).
		\end{array}
		\end{equation*}
		For simplicity, we assume $\eta_e=\eta$,~$\tilde{\eta}_e=\tilde{\eta}$ and ${v}_1={v}_2=p_1=p_2=c_2=1$.
		Hence, to ensure $S_2+Z_2\leq0$, the feedback law $u_1^k=f_{1,-1}^k$ from assumption \eqref{ineq} has to fulfill:
		\begin{align}
		\nonumber (u_{1}^k)^2&\leq(f_{1,N-1}^k)^2\exp(-\eta x_{1,N-1})+(f_{2,N-1}^k)^2\exp(-\eta x_{2,N-1})\\
		\nonumber &-(\mu_2)^2\exp(-\eta x_{2,0})-2q_2^k(f_{1,N-1}^k-\mu_2)\exp(-\tilde{\eta}t_k)\exp(-\tilde{\eta}\tau)\\
		&-\tau\big(f_{1,N-1}^k-\mu_2\big)^2\exp(-\tilde{\eta}t_k)\exp(-\tilde{\eta}\tau)\eqqcolon Y^k,
		\label{eq:u1k}
		\end{align}
	\end{enumerate}
or, equivalently, if $Y^k\geq0$, $u_1^k\leq\sqrt{Y^k}$.
	At the beginning, if queue 2 is not equal to zero yet, we use the time-dependent control law $Y^k$ \eqref{eq:u1k}. As soon as queue 2 is damped to 0, we can consider case 1 again. That means that for $k\rightarrow (K-1)$, we are able to determine a matrix $G$, and thus a $\kappa>0$, such that exponential stability is ensured. 
	Consequently, it is not possible to find a general $\kappa>0$ for all $t_k$ such that exponential stability is obtained. Quite the contrary, to include a linear feedback law, we have to differentiate between the two cases above. This explains the kink in the Lyapunov function when only the linear feedback (LF) with $\kappa>0$ for all time steps is used, see Figure \ref{img:lyapcompl}. 
	
In the following, we call the feedback law that differentiates between the non-queue case and the case with queues the {\em mixed feedback (MF)}.
We note that $Y^k$ in $\eqref{eq:u1k}$ 	coincides with the continuous version $X(t)$ in \eqref{eq:X_t} (under the assumption that $q_2^k>0$ and ${v}_e=1$) as $\tau\rightarrow0$. 
A similar result is obtained for the numerical decay rate $\nu=\min(\nu_1,\nu_2)$ (see \eqref{eq:nu0}) of the discrete Lyapunov function:
\begin{cor}\label{cor}
Assume $\tau = \frac{h}{\max_e(\tilde{v}^e)}$, i.e.\ CFL with equality, is satisfied. Then the numerical decay rate $\nu=\min(\nu_1,\nu_2)$ converges to the analytical decay rate $\min(\eta,\tilde{\eta})\min_e({v}_e)$ as $h \to 0$
(see Lemma \ref{lem:conditionu}).
\end{cor}
\begin{proof}
Since the minimum is continuous, we show the convergence of $\nu_1 \to \eta \min_e({v}_e)$ and $\nu_2 \to \tilde{\eta} \min_e({v}_e)$ separately.
We use a the Taylor expansion $Tf(h;0)=\eta_e h+O(h^2)$ for $f(h)\coloneqq1-\exp(-\eta_eh)$ and obtain
	\begin{align*}
	\nu_1& = \min_e\left(\tilde{v}^e\frac{1}{h}Tf(h;0)\right) = \min_e\left(\tilde{v}^e\Big(\eta_e+O(h)\Big)\right)\overset{h\rightarrow 0}{\longrightarrow} \min_e(\tilde{v}^e\eta_e).
	\end{align*}
Using the CFL condition with equality, we have 
	\begin{align*}
	\nu_2&=\frac{1}{h}\max_e(\tilde{v}^e)\min_e(1-\exp(-\tilde{\eta}_e\tilde{v}^e\tau)) =\frac{1}{h}\max_e(\tilde{v}^e)\min_e\left(1-\exp\Big(-\tilde{\eta}_e\tilde{v}^e\frac{h}{\max_e(\tilde{v}^e)}\Big)\right).
	\end{align*}
	We use the Taylor expansion $ Tf(h;0)=\frac{\tilde{v}^e}{\max_e(\tilde{v}^e)}\tilde{\eta}_e h+O(h^2) $ for $f(h)\coloneqq1-\exp\left(-\tilde{\eta}_e\tilde{v}^e\frac{h}{\max_e(\tilde{v}^e)}\right)$. This yields
	\begin{align*}
	\nu_2& =\min_e\left(\frac{1}{h}\max_e(\tilde{v}^e)\Big(1-\exp\Big(-\tilde{\eta}_e\tilde{v}^e\frac{h}{\max_e(\tilde{v}^e)}\Big)\Big)\right)\\
	& = \min_e\left(\tilde{v}^e\big(\tilde{\eta}_e+O(h)\big)\right)\\
	\overset{h\rightarrow 0}{\longrightarrow} & ~~~\min_e(\tilde{v}^e\tilde{\eta}_e).
	\end{align*}
\end{proof}

\subsection{Computational Experiments}
We stick to the example presented in Figure~\ref{img:ginout}.
First, we consider case 1, which means that we have to ensure that the matrix $G^TP_0\Lambda G-P_{N-1}\Lambda$ is negative semidefinite. Since we consider a serial network, the matrix $G$ is given by \eqref{matrixG}.
The two eigenvalues of $G^TP_0\Lambda G-P_{N-1}\Lambda$ are:
\begin{align*}
\lambda_1&=p_2{v}_2\exp(-\eta x_{2,0})-p_1{v}_1\exp(-\eta x_{1,N-1}),\\
\lambda_2&=\kappa^2p_1{v}_1\exp(-\eta x_{1,0})-p_2{v}_2\exp(-\eta x_{2,N-1}).
\end{align*}
Assuming $p_1=p_2={v}_1={v}_2=1$, due to Theorem \ref{theo:discr}, we further have to consider the offset $x_{1,N-1}=x_{2,0}=l$ to guarantee $\lambda_1\leq0$:
\begin{align*}
\lambda_1&=\exp(-\eta l)-\exp(-\eta l)=0,\\
\lambda_2&=\kappa^2-\exp(-\eta 2l).
\end{align*}
To ensure the exponential stability, the eigenvalue $\lambda_2$ has to be nonpositive. Therefore, the linear feedback law is given by:
\begin{equation*}
u_1^k=\kappa f_{2,N-1}^k,
\end{equation*} 
with $\kappa>0$ satisfying
\begin{equation}
\kappa^2\leq\exp(-\eta2l) \; \Leftrightarrow \; \kappa\leq\exp(-\eta l).
\label{eq:kappabed}
\end{equation} 

Note that the above case is only valid if the queue in front of processor 2 is equal to zero. If this is not the case, we implement the feedback law given by $Y^k$ from \eqref{eq:u1k}. Due to the offset in the spatial domain we obtain:
\begin{equation*}
\begin{array}[b]{r@{}c@{}l}
(u_{1}^k)^2&\leq&(f_{1,N-1}^k)^2\exp(-\eta x_{1,N-1})+(f_{2,N-1}^k)^2\exp(-\eta x_{2,N-1})-(\mu_2)^2\exp(-\eta x_{2,0})\\
&-&2q_2^k(f_{1,N-1}^k-\mu_2)\exp(-\tilde{\eta}t_k)\exp(-\tilde{\eta}\tau)-\tau\big(f_{1,N-1}^k-\mu_2\big)^2\exp(-\tilde{\eta}t_k)\exp(-\tilde{\eta}\tau)\\
&\overset{\text{offset}}{=}&(f_{1,N-1}^k)^2\exp(-\eta l)+(f_{2,N-1}^k)^2\exp(-\eta 2l)-(\mu_2)^2\exp(-\eta l)\\
&-&2q_2^k(f_{1,N-1}^k-\mu_2)\exp(-\tilde{\eta}t_k)\exp(-\tilde{\eta}\tau)-\tau\big(f_{1,N-1}^k-\mu_2\big)^2\exp(-\tilde{\eta}t_k)\exp(-\tilde{\eta}\tau)\\
&=&Y^k.
\end{array}
\end{equation*}

So a possible choice for the mixed feedback law is given by:
\begin{equation*}
u_1^k=\begin{cases}0,~~~~~~~~~\text{if}~ Y^k<0,\\\sqrt{Y^k},~~~~\text{if}~Y^k\geq 0.\end{cases}
\label{eq:Ykprac}
\end{equation*}
For the numerical computations, if not stated otherwise, we choose the parameter values $\eta=\tilde{\eta}=\frac{1}{2},l=\frac{1}{2},h=0.01,T=30$ and $\frac{\tau}{h}\max_e{v}_e=1$. From \eqref{eq:kappabed}, we choose $\kappa=e^{-\frac{1}{4}}\approx0.7788$. 
We set $\mu_1=6,\mu_2=4$ and choose the initial conditions $f_1(0,x)=4,f_2(0,x)=4$ and $q_1(0)=0,q_2(0)=1$. The velocities are ${v}_e=1$ for $e\in\{1,2\}$.
We denote the Lyapunov function, i.e.\ the upper bound in \eqref{eq:Vup1}, by:
\begin{equation}
V_{up}^k\coloneqq\exp(-\nu\tau k)V^0.
\label{eq:Vupdiscr}
\end{equation} 
In the first numerical example, we compare the linear feedback control from Section \ref{sec:analytical}, with the mixed feedback control from Theorem \ref{theo:discr} so that there is a difference between the non-queue case and the case with positive queues.
Figure \ref{img:ex1lf} shows a decreasing Lyapunov function with a kink that results from using the linear feedback law with $\kappa=0.5$. 
Apparently, the discrete Lyapunov function lies underneath the upper bound $V_{up}$. However, the kink in the Lyapunov function can not be avoided. Due to the occurrence of the kink, the parameter value $\kappa=0.7788$ is not appropriate for the linear feedback since it would lead to a Lyapunov function that lies above $V_{up}$. For all values of $\kappa$, we can at most expect a decreasing discrete Lyapunov function, i.e., asymptotic stability.

For the mixed feedback law, the parameter value $\kappa=0.7788$ leads to the desired exponential decay of the Lyapunov function, see Figure \ref{img:ex1mf}, and $V^k$ and $V_{up}^k$ overlay almost exactly. The difference in the two approaches becomes even more clear if we have a look at the log-plot of both discrete Lyapunov functions. The log-plot of the discrete Lyapunov function of the mixed feedback is a straight line, whereas we can notice a kink in the log-plot for the linear feedback in Figure \ref{img:ex1lflog}. 
\begin{figure}[htb!]
	\centering
	\begin{subfigure}[b]{.49\textwidth}
		\centering
		\includegraphics[width=\textwidth]{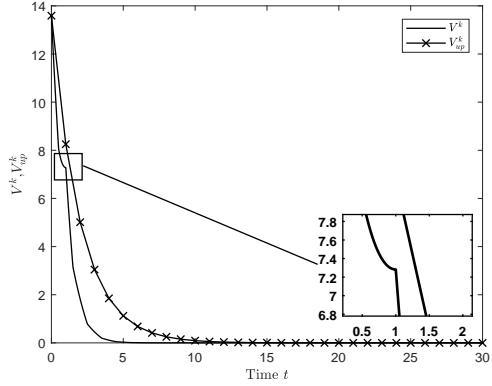}	
			\caption{Lyapunov function for the linear feedback (LF)}
			\label{img:ex1lf}
	\end{subfigure}
	\begin{subfigure}[b]{.49\textwidth}
		\includegraphics[width=\textwidth]{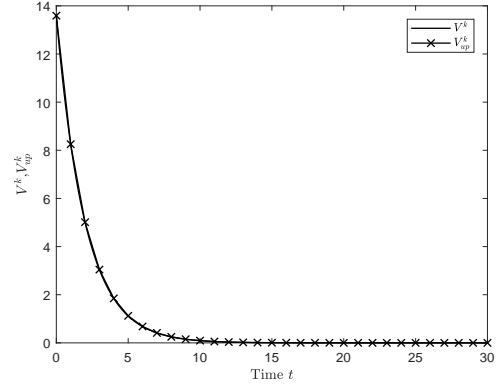}
			\caption{Lyapunov function for the mixed feedback (MF)}
			\label{img:ex1mf}
	\end{subfigure}
	\caption{Discrete Lyapunov functions for linear and mixed feedback}
\end{figure}
\begin{figure}[htb!]
	\centering
	\begin{subfigure}[b]{.49\textwidth}
		\centering
		\includegraphics[width=\textwidth]{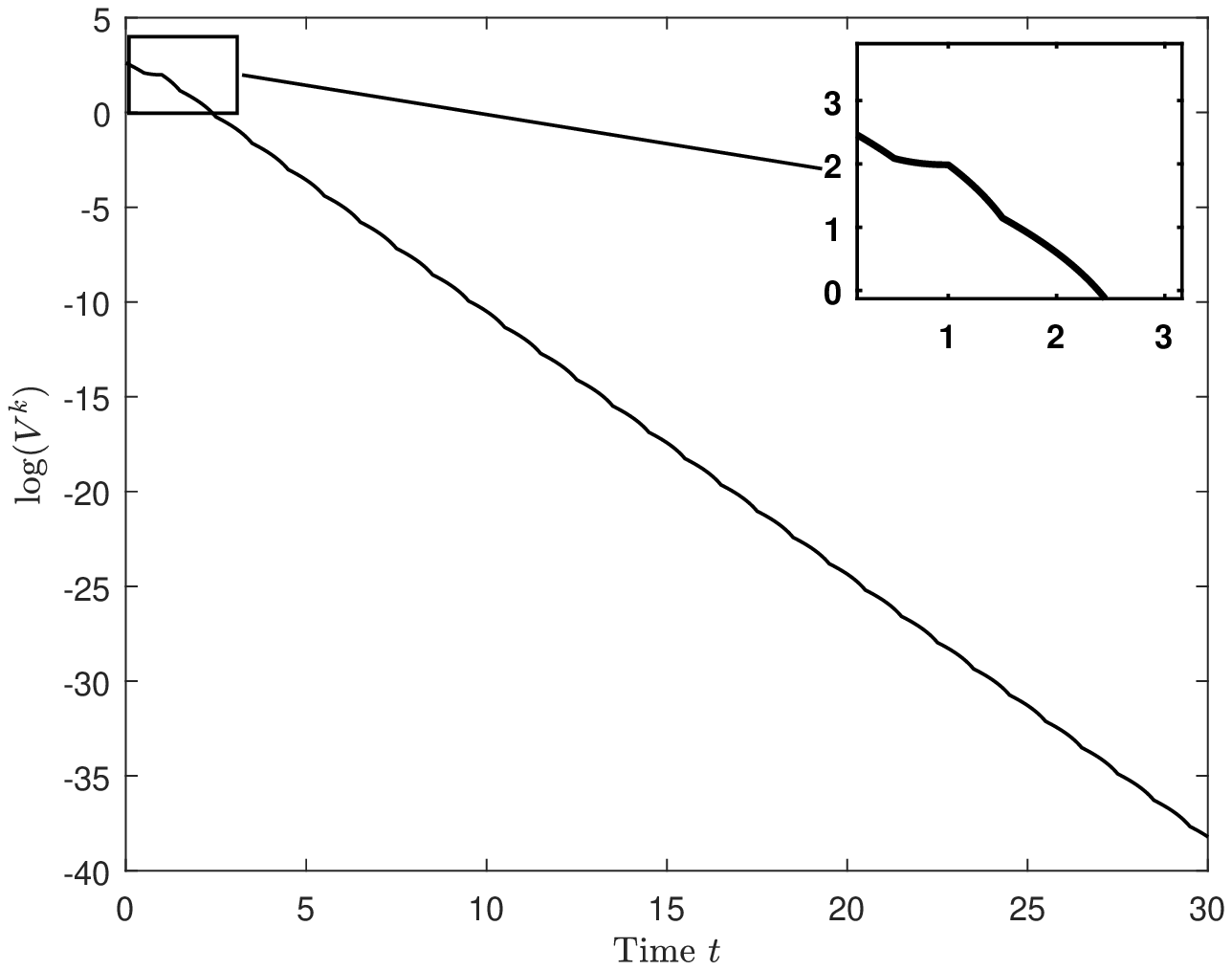}
			\caption{Log-plot of the Lyapunov function (LF)}
			\label{img:ex1lflog}
	\end{subfigure}
	\begin{subfigure}[b]{.49\textwidth}
		\includegraphics[width=\textwidth]{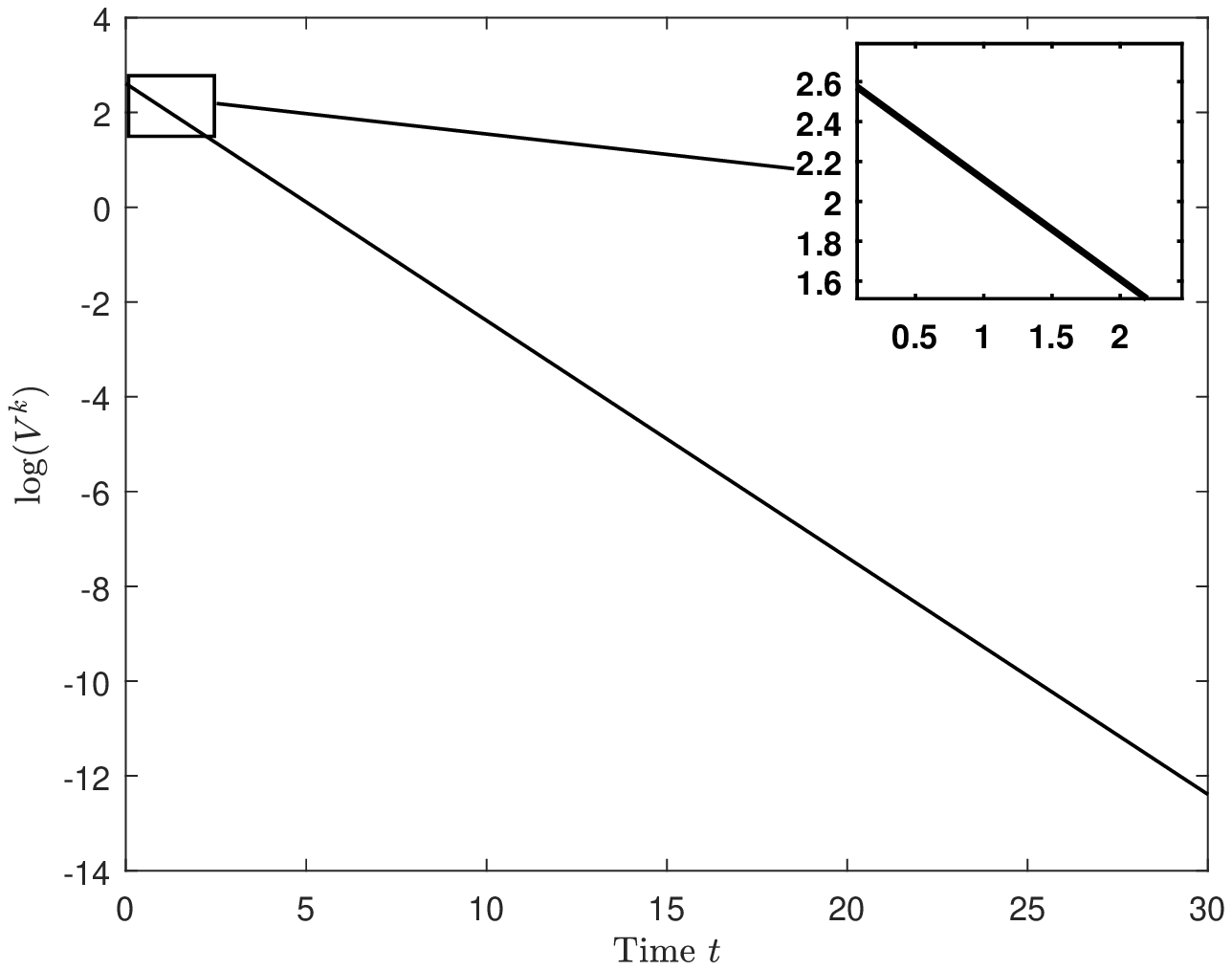}	
			\caption{Log-plot of the Lyapunov function (MF)}
			\label{img:ex1mflog}
	\end{subfigure}
	\caption{Discrete Lyapunov functions for linear and mixed feedback: log-plot}
\end{figure}

For the same setting, we investigate the numerical behavior of the decay rate that is given by:
\begin{equation*}
 \nu=\min\Bigg\{\frac{1}{h}\min_e({v}_e(1-\exp(-\eta_eh))),\frac{1}{h}\max_e({v}_e)\min_e(1-\exp(-\tilde{\eta}_e{v}_e\tau))\Bigg\}. 
 \end{equation*}
 To do so, we set $\eta=\tilde{\eta}=0.575$  and vary the value for the velocities. Table \ref{nu} shows the convergence of the decay rate $\nu$ to $\min_e({v}_e)\min(\eta,\tilde{\eta})$ for ${v}_e=1$ and ${v}_e=\frac{1}{2}$ as mentioned in Corollary \ref{cor}. Furthermore, the table entries 
$\lVert \cdot\rVert_{\infty}$ and $\lVert \cdot\rVert_{L^2}$ denote the $L^\infty$- and $L^2$-norm of the difference $V_{up}^k-V^k$. As expected, we observe first-order convergence of the discrete Lyapunov function towards the upper bound.
\renewcommand{\arraystretch}{1.2}
\newcolumntype{C}[1]{>{\centering\arraybackslash}m{#1}}
\begin{table}[H]
	\centering
	\begin{threeparttable}
		\begin{tabular}{|C{0.6cm}||C{0.9cm}|C{0.8cm}|C{0.9cm}|C{0.8cm}|C{0.9cm}||C{0.9cm}|C{0.8cm}||C{0.9cm}|C{0.8cm}|C{0.9cm}|}
			\hline
			\multicolumn{1}{|c||}{}  &
			\multicolumn{5}{|c||}{${v}_e=1$} & \multicolumn{5}{|c|}{${v}_e=0.5$}\\
			\hline
			$N$ &$\lVert \cdot\rVert_{\infty}$ & Conv. Rate& $\lVert \cdot\rVert_{L^2}$& Conv. Rate & $\nu$ & $\lVert \cdot\rVert_{\infty}$ & Conv. Rate & $\lVert \cdot\rVert_{L^2}$& Conv. Rate & $\nu$ \\
			\hline
			\hline	
			10 &0.0754 & - &0.1326& -  & 0.5668  &0.0834 &-& 0.2007 &-& 0.2834 \\
			\hline
			50 &0.0151& 0.99 &0.0265& 1.00 & 0.5734 &0.0153 &1.05 & 0.0380 &1.03&  0.2867\\
			\hline
			100 &0.0075& 1.01 &0.0132& 1.00 &  0.5742 &0.0076 &1.01& 0.0189 & 1.01& 0.2871\\
			\hline
			200 & 0.0038& 0.99&0.0066&1.00&  0.5746 &0.0038 & 1.00 &0.0094  & 1.01&0.2873 \\
			\hline
			400 &0.0019&1.00&0.0033& 1.00 & 0.5748 &0.0019& 1.00 & 0.0047  & 1.00& 0.2874 \\
			\hline
			800 & 0.0009&1.05&0.0017&0.97 & 0.5749 & 0.0009&1.08 & 0.0023 & 1.03&  0.2874 \\
			\hline
		\end{tabular}
		\caption{Convergence of the decay rate $\nu$ for $\eta=\tilde{\eta}=0.575$ and first-order convergence of the discretization for a CFL constant equal to 1,~$N=\frac{1}{2h}$ and different velocities}
		\label{nu}
	\end{threeparttable}
\end{table}
The dependence of the discrete Lyapunov function on the parameter $\kappa$ is shown in Table \ref{frac}. Here, the ratio of the Lyapunov function at time $t=T$ and $t=0$ is small for small values of $\kappa$ and vice versa. The spatial step size $h=0.00125$ and ${v}_e={v}=1$ imply that the decay rate $\nu$ is quite close to  $\tilde{\eta}=\min(\eta,\tilde{\eta})$, see Table \ref{frac}.
\begin{table}[H]
	\centering
	\begin{tabular}{|C{1.5cm}|C{1.5cm}|C{1.5cm}|C{1.5cm}|C{1.5cm}|}
		\hline
		$\kappa$ & $\frac{V^T}{V^0}$ & $\eta$ & $\tilde{\eta}$ & $\nu$ \\
		\hline
		\hline	
		0.1 &$3.75e^{-60}$  & 4.6052    & 0.5752 & 0.5750 \\
		\hline
		0.25 & $1.40e^{-36}$ & 2.7726 & 0.5752 & 0.5750\\
		\hline
		0.5 & $1.05e^{-18}$ & 1.3863 & 0.5752 & 0.5750\\
		\hline
		0.75 &  $3.20e^{-8}$ &0.5752  & 0.5752 & 0.5750\\
		\hline
	\end{tabular}
	\caption{Dependence of the Lyapunov function on $\kappa$}
	\label{frac}
\end{table}

However, considering production systems with only decreasing queues does not seem reasonable. Therefore, we set $\mu_1=6$, $\mu_2=4$ and choose the initial conditions $f_1(0,x)=6$, $f_2(0,x)=4$ and $q_2(0)=0$, which lead to an increasing queue in front of processor 2. We set ${v}_1={v}_2=1$, $l=\frac{1}{2}$, $T=30$ and $\eta=\tilde{\eta}=0.2$. The Lyapunov function $V_1$ for the flux and $V_2$ for the queue are depicted in Figure \ref{img:2}. 
We observe exponential decay of the composed Lyapunov function $V=V_1+V_2$, see Figure \ref{img:1}. Moreover, the outflow and the queue of processor 2 are damped to zero, see Figure \ref{img:3} and \ref{img:4}. The linear feedback takes over as soon as the queue is equal to zero. The reason for the shape of the feedback becomes more clear in the following example.

\begin{figure}[H]
	\centering
	\begin{subfigure}[b]{0.49\textwidth}
		\centering
		\includegraphics[width=\textwidth]{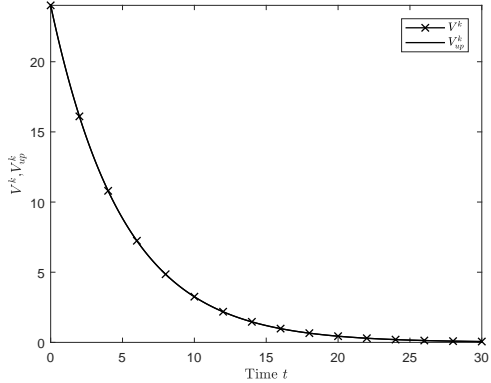}
		\caption[Network2]%
		{{Discrete Lyapunov function $V^k$ and upper bound $V_{up}^k$, see \eqref{eq:disclya1} and \eqref{eq:Vupdiscr}}}    
		\label{img:1}
	\end{subfigure}
	\hfill
	\begin{subfigure}[b]{0.49\textwidth}  
		\centering 
		\includegraphics[width=\textwidth]{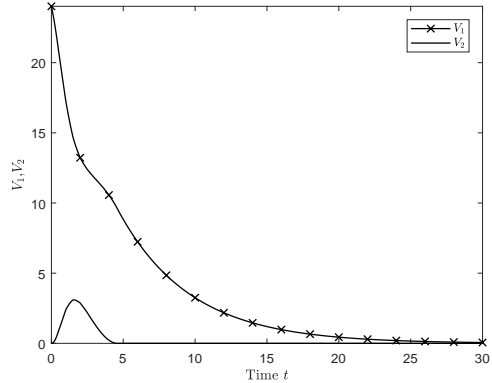}
		\caption[]%
		{{Lyapunov function for the flux $V_1$ and for the queue $V_2$, see \eqref{eq:lyap}}}    
		\label{img:2}
	\end{subfigure}
	\begin{subfigure}[b]{0.49\textwidth}   
		\centering 
		\includegraphics[width=\textwidth]{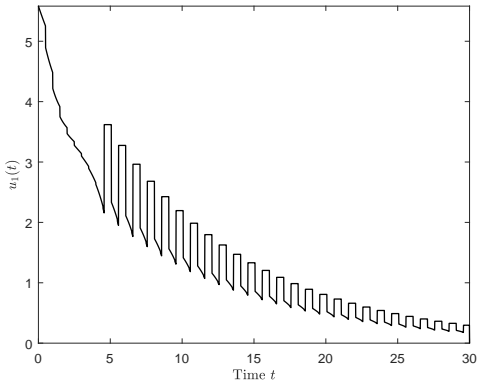}
		\caption[]%
		{{Feedback $u_1(t)$, see \eqref{eq:u1k}}}    
		\label{img:3}
	\end{subfigure}
\begin{subfigure}[b]{0.49\textwidth}   
	\centering 
	\includegraphics[width=\textwidth]{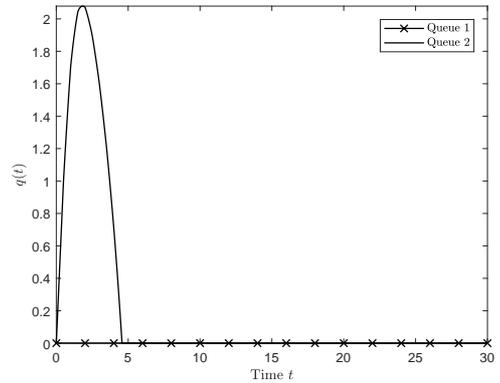}
	\caption[]%
	{{Queue network with an increasing queue 2}}    
	\label{img:4}
\end{subfigure}
\caption[ The average and standard deviation of critical parameters ]
{Production system with one increasing queue} 
\label{img:LyapincrSuplot}
\end{figure}

\newpage
As a final experiment, we investigate the behavior of the Lyapunov function for a larger gap between the two maximal capacities $\mu_1$ and $\mu_2$. Therefore, we set $\mu_2=4$ and vary the values for $\mu_1$. We choose the initial conditions for the flux equal to the maximal capacities and assume again $q_e(0)=0$. 
As in the previous example, we set $\eta=0.2$, $\tilde{\eta}=0.2$, which leads to $\kappa=e^{-\frac{1}{10}}\approx0.9048$. Regarding the linear feedback law, a higher $\mu_1$ leads to a bigger kink in the Lyapunov function as shown in Figure \ref{img:LyapincrSuplot1}. In comparison to that, we notice that the mixed feedback law drastically reduces the kink in the Lyapunov function and leads to the desired exponential decay. This shows that the mixed feedback law is even able to deal with higher queue loads. 

In Figure \ref{img:LyapincrSuplot11}, the feedback $u_1^k$ for both feedback laws is depicted. It should be noted that the mixed feedback $u_1^k$ is decreasing more slowly as soon as the queue load is sinking (i.e. $\mu_2>f_{1,N-1}^k$). The evolution of the queue is reflected in the feedback since it is part of the mixed feedback law, see \eqref{eq:u1k}. This behavior becomes even more significant for larger $\mu_1$ leading to a a higher queue load. 
Furthermore we remark that the mixed feedback slightly increases for all choices of $\mu_1$ as soon as the linear feedback takes over due to the queue load being equal to zero. Shortly after the increase, the linear feedback starts to decrease again since the flux in the first processor has been already damped below $\mu_2$.

Applying simply the linear feedback law leads to a constant feedback $u_1^k=\kappa f_{2,N-1}^k=4\kappa$ at the beginning due to the queue being nonzero. 
The constant behavior of $u_1^k$ remains maintained longer for larger $\mu_1$, also resulting in a bigger kink of the Lyapunov function. Finally, both feedback laws damp the outflow and the queue to zero.

\begin{figure}[htb!]
	\centering
	\begin{subfigure}[b]{1\textwidth}
		\centering
		\includegraphics[width=\textwidth]{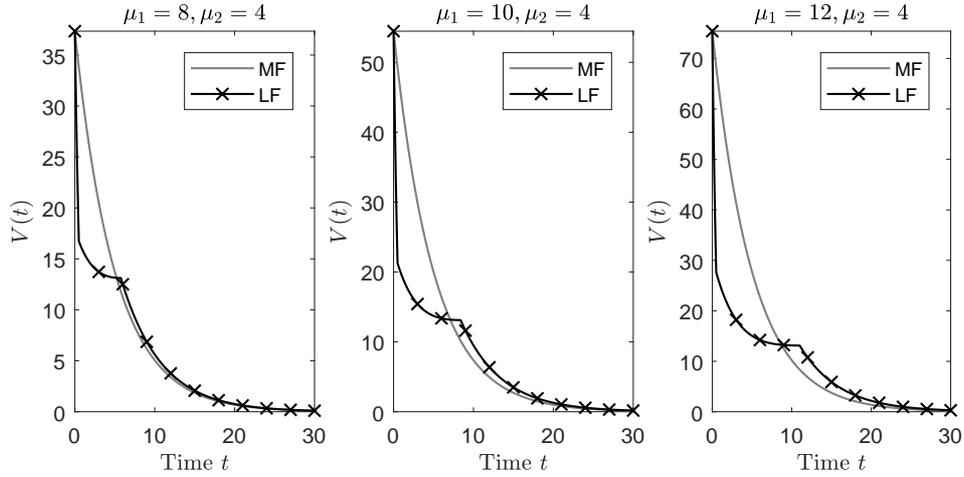}
		\caption{Lyapunov function for the MF and LF for different maximal capacities}
		\label{img:LyapincrSuplot1}
	\end{subfigure}
	\vskip\baselineskip
	\begin{subfigure}[b]{1\textwidth}
		\includegraphics[width=\textwidth]{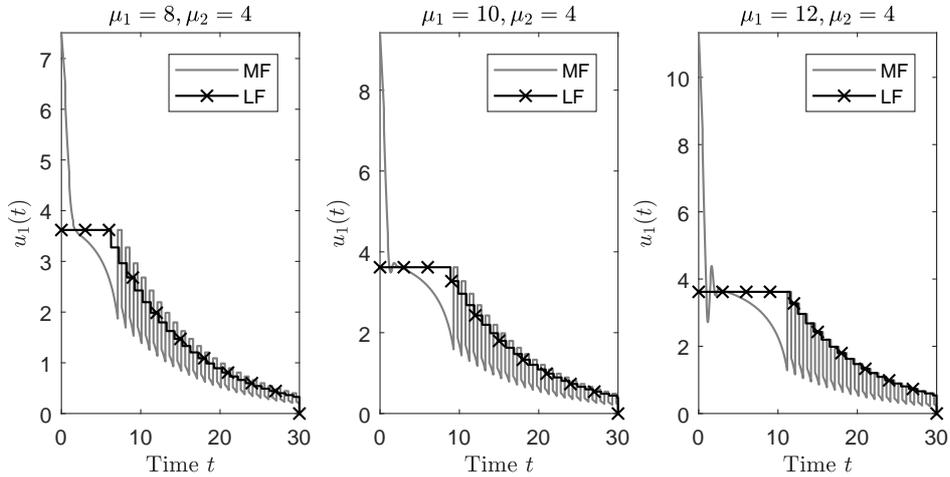}
		\caption{Feedback for the MF and LF for different maximal capacities}
		\label{img:LyapincrSuplot11}
	\end{subfigure}
	\caption{Lyapunov functions and feedback laws}
\end{figure}

\section{Conclusion}
The main topic of this paper is the continuous and discrete feedback stabilization 
for a coupled PDE-ODE system describing a serial production system. 
By using of an adapted Lyapunov function, we are able to prove exponential stability 
and derive a linear feedback law. We also observe that the linear feedback law leads to at most asymptotic stability. 
Based on the numerical approximation for the coupled PDE-ODE system, we can compute decay rates
and derive a mixed feedback law which is consistent with the analytical result but able to resolve queueing situations. 

\section*{Acknowledgments}
This work was financially supported by BMBF project ENets (05M18VMA) and the DFG grants No. GO 1920/4-1 and GO 1920/7-1.

\bibliography{quellen}
\bibliographystyle{siam}

\end{document}